\documentclass[12pt]{amsart}

\usepackage{latexsym,url}
\usepackage{amsthm}
\usepackage{amsmath}
\usepackage{amsfonts}
\usepackage{amssymb}
\usepackage[dvips]{graphicx}
\usepackage{xypic}
\usepackage{enumerate}

\input xy
\xyoption{all}

\theoremstyle{plain}
\newtheorem{theo}{Theorem}[section]
\newtheorem{lemma}[theo]{Lemma}

\newtheorem{propo}[theo]{Proposition}
\newtheorem{coro}[theo]{Corollary}
\theoremstyle{definition}
\newtheorem{defi}[theo]{Definition}
\newtheorem{rem}[theo]{Remark}

\newcommand\id{\operatorname{id}}

\newcommand\cg{\mathcal {G}}

\newcommand\ck{\mathcal {K}}

\newcommand{\CPO}{\operatorname{\bf CPO}}

\date{June 16, 2025}
 
\begin{document}

\title[Are chain-complete posets co-wellpowered?]
{Are chain-complete posets co-wellpowered?}

\author[J. Jurka and J. Rosick\'{y}]
{J. Jurka and J. Rosick\'{y}}

\thanks{Both authors were supported by the Grant Agency of the Czech Republic under the grant 
              22\nobreakdash-02964S} 

\address{
\newline J. Jurka\newline
Department of Mathematics and Statistics\newline
Masaryk University, Faculty of Sciences\newline
Kotl\'{a}\v{r}sk\'{a} 2, 611 37 Brno, Czech Republic\newline
jan.jurka@mail.muni.cz
\newline J. Rosick\'{y}\newline
Department of Mathematics and Statistics\newline
Masaryk University, Faculty of Sciences\newline
Kotl\'{a}\v{r}sk\'{a} 2, 611 37 Brno, Czech Republic\newline
rosicky@math.muni.cz
}

\begin{abstract}
We show that the category $\CPO$ of chain-complete posets is nearly locally presentable, we give a characterization of strong epimorphisms in $\CPO$, and we offer an alternative proof of co-wellpoweredness of $\CPO$.
\end{abstract} 


\maketitle

\section{Introduction}

Chain-complete posets play an important role in theoretical computer science, in denotational semantics and domain theory. The category 
$\CPO$ of chain-complete posets is complete and cocomplete \cite{M}.
Epimorphisms in $\CPO$ are described in \mbox{\cite[Section 3, Theorem]{P}}
and the same paper shows that $\CPO$ is co-wellpowered. A well-known deficiency of chain complete posets is that the category $\CPO$ is not locally presentable (see \cite[1.18(5)]{AR}). In fact, no non-empty chain-complete poset is presentable (see \cite[1.14(4)]{AR}). The reason is that directed colimits produce new chains whose joins  have to be added. But this does not happen for coproducts and every finite chain-complete poset is nearly finitely presentable in the sense of \cite{PR}. We will show that $\CPO$ is nearly locally finitely presentable. Although this was claimed in \cite{PR}, the proof there was based on a wrong belief that epimorphisms are one-step dense under directed joins (and that the two-element chain $2$ is a strong generator). Even, epimorphisms are not iterated dense under directed joins but they are dense under all joins (which implies that $\CPO$ is co-wellpowered). This corrects our paper \cite{JR} (see \cite{JR_err}).

We will show that the three element chain $3$ is a strong generator
in $\CPO$. We will also give a direct proof of co-wellpoweredness
of $\CPO$ avoiding a characterization of epimorphisms. In addition,
we will give a characterization of strong epimorphisms in $\CPO$. 
The category $\CPO$ has both (epimorphism, strong monomorphism) and (strong epimorphism, monomorphism) factorizations (see \cite[0.5]{AR}). Every regular epimorphism is strong but there are strong epimorphisms which are not regular. In fact, $\CPO$ cannot have (regular epimorphism, mo\-no\-morphism) factorizations because then it would be locally presentable
provided that Vop\v enka's principle is assumed (see \cite{PR}).

Let us recall that a chain-complete poset is a poset where every chain has a join. One often calls them complete partial orders, briefly cpo's. Morphisms in the category $\CPO$ of cpo's are maps preserving joins of chains. We will call these morphisms {\it cpo maps}. Every directed subset of a cpo has a join and morphisms of cpo's preserve directed joins (see \cite[1.7]{AR}). Since we allow empty chains, every cpo has the smallest element $0$ and cpo maps preserve $0$. Hence coproducts are disjoint unions with the smallest elements identified.

\vskip 3 mm

\noindent {\bf Acknowledgement.} We are grateful to J. Ad\'amek for pointing out \cite{LP,P} to us.

\section{$\CPO$ has a strong generator}
A set $\cg$ of objects in a category $\ck$ is a \textit{generator} provided that for each pair $f,g\colon A\to B$ of distinct morphisms there exists an object $G$ in $\cg$ and a morphism $h\colon G\to A$ such that $fh\neq gh$. A generator $\cg$ is called \textit{strong} provided that for each object $B$ and each proper subobject of $B$ there exists a morphism $G\to B$ with $G$ in $\cg$ which does not factorize through that subobject. The following characterization of strong generators holds in cocomplete categories (see \cite[0.6]{AR}): A set $\cg$ of objects is a strong generator iff every object of the category is an extremal quotient of coproduct of objects from $\cg$.

\begin{rem}\label{mono}
\begin{enumerate}[(1)]
\item Monomorphisms in $\mathbf{CPO}$ are precisely the injective cpo maps and isomorphisms are precisely the order-reflecting bijections. This means that proper subobjects in $\mathbf{CPO}$ are precisely the injective cpo maps that are either not surjective, or they are surjective but not order-reflecting.
\item Two-element chain $2$ is clearly a generator in $\CPO$. But it is not strong (as it is claimed in \cite{PR}). Indeed, consider
a three-element cpo $A$ with two incomparable elements $1$ and $2$ (and the smallest element $0$). Let $3$ be the three-element chain $0<1<2$. Then the identity $\id$ on $\{0,1,2\}$ makes $A$ a proper subobject of $3$ but every cpo map $2\to 3$ factorizes
through $\id$.
\end{enumerate}
\end{rem}

\begin{lemma}
Three-element chain $3$ forms a strong generator in $\mathbf{CPO}$.
\end{lemma}
\begin{proof}
Obviously, $3$ is a generator. Suppose that $f\colon A\to B$ is a proper subobject in $\CPO$. If $f$ is not surjective there exists an element $b \in B - f[A]$, and thus the cpo map $3 \to B$ sending $1$ and $2$ to $b$ does not factorize through $f$. If $f$ is surjective but not order-reflecting there exist elements $a_1, a_2 \in A$ incomparable in $A$ but $a_1<a_2$ in $B$.
Then the cpo map $3\to B$ sending $1$ to $a_1$ and $2$ to $a_2$ does not factorize through $A$.
\end{proof}

\section{$\CPO$ is co-wellpowered}
Following \cite[Section 4, Corollary 2]{P}, $\CPO$ is co-wellpowered.\footnote{Note that the result in \cite{P} is stated for a subset system $Z$. To obtain the result for chain-complete posets we choose $Z(P)$ to be the set of all directed subsets of a poset $P$.} In this section, we offer a direct proof avoiding a characterization of epimorphisms.  

\begin{defi}
Let $X \subseteq A$ be a subset of a cpo $A$. Then
\begin{itemize}
\item a {\it closure $\operatorname{dcl}(X)$ of $X$ under directed joins} is the smallest subset of $A$ that is closed under directed joins and contains $X$,
\item a cpo map $f \colon A \to B$ is called {\it dense} if $B$ is the closure of the image $f[A]$ under directed joins, and
\item a {\it closure $\operatorname{cl}(X)$ of $X$ under all joins} is the smallest subset of $A$ that is closed under all joins and contains $X$.
\end{itemize}
\end{defi}

\begin{rem}
The directed join closure of $X \subseteq A$ can be constructed by transfinite induction as follows:
\begin{align*}X_0&= X, \\ X_{\alpha + 1}&= \big\{\bigvee Y \mid Y \subseteq X_\alpha \text{ is directed in $X_\alpha$}\big\}, \\ X_\alpha &= \bigcup_{\beta < \alpha} X_\beta \text{ for each limit ordinal $\alpha$}.\end{align*} 
It is obvious that $\operatorname{dcl}(X) = X_{|A|^+}$.
\end{rem}

\begin{lemma}\label{lem-epi}
Let $f \colon A \to B$ be a morphism in $\CPO$, then
$$\operatorname{dcl}(f[A]) = B \quad \Longrightarrow \quad f\ \text{is an epi} \quad \Longrightarrow \quad \operatorname{cl}(f[A]) = B.$$
\end{lemma}
\begin{proof}
The first implication follows easily by transfinite induction. We prove the second implication. Suppose that $f\colon A\to B$ is a morphism in $\CPO$. Consider the factorization
$$
A \xrightarrow{\ e } C \xrightarrow{\ m } B 
$$
of $f$ through $C := \operatorname{cl}(f[A])$. Note that $e$ is given by the codomain-restriction and $m$ is the inclusion. Since $C$ is defined to be the closure in $B$ under all joins, it is a cpo, and the maps $e$ and $m$ are cpo maps. Consider the pushout
$$
\xymatrix@C=4pc@R=3pc{
C \ar [r]^{m} \ar [d]_{m}& B \ar [d]^{g}\\
B \ar [r]_{h}& D
}
$$
in $\CPO$. We claim that $g \neq h$ if $C \neq B$ because $g$ and $h$ differ on each element of $B \setminus C$.  

Let us first explicitly describe $D$, $g$, and $h$. Consider the set $C \cup B_1 \cup B_2$, where $B_i = \{(b,i) \mid b \in B \setminus C\}$ for $i \in \{1,2\}$. We will denote by $\pi_1$ the projection onto the first component. Consider a binary relation $\lesssim$ on $D$ such that $x \lesssim y$ if and only if 
\begin{itemize}
\item $x, y \in C, x \leq_B y$ or
\item $x \in C, y \in B_1 \cup B_2, x \leq_B \pi_1(y)$ or 
\item $x \in B_1 \cup B_2, y \in C, \pi_1(x) \leq_B y$ or 
\item $x, y \in B_1, \pi_1(x) \leq_{B} \pi_1(y)$ or
\item $x, y \in B_2, \pi_1(x) \leq_{B} \pi_1(y)$ or
\item $x = (b,1)$, $y = (d,2)$ or $x = (b,2)$, $y = (d,1)$, where $b, d \in B \setminus C$, and there exists $b' \geq_B b$ and a subset $Z \subseteq B \setminus C$ such that $b' \in \operatorname{dcl}(Z)$ in $B$, and a set $\{c_z \in C \mid z \in Z\}$ of elements such that $z \leq_B c_z \leq_B d$ for each $z \in Z$.
\end{itemize}
The relation $\lesssim$ is obviously reflexive. It is also easily seen to be transitive. Indeed, the last condition ensures that if $(b, 1) \lesssim c \lesssim (d, 2)$, where $c \in C$, then $(b,1) \lesssim (d,2)$. Moreover, again by the last condition, if $(b,1) \lesssim (d,2) \lesssim (e,1)$, then ${b \leq_B d \leq_B e}$, and thus $(b, 1) \lesssim (e,1)$. Also, by the last condition, if ${(b,1) \lesssim (e,1) \lesssim (d,2)}$, then ${b \leq_B e \leq_B e'}$, and thus $(b,1) \lesssim (d,2)$. Analogous arguments demonstrating transitivity also work if we switch the roles of indices $1$ and $2$. Note that because of the last condition the relation $\lesssim$ is not necessarily antisymmetric. Thus, let $(D,\leq)$ be the antisymmetric quotient of ${(C \cup B_1 \cup B_2,\lesssim)}$. The maps $g$ and $h$ can be described as follows:
$$
g(b) =
\begin{cases}
b, \qquad\; \text{ if } b \in C,\\
[(b,1)], \text{ if }b \in B \setminus C,\\
\end{cases}
\quad
h(b) =
\begin{cases}
b, \qquad\; \text{ if } b \in C,\\
[(b,2)], \text{ if }b \in B \setminus C.\\
\end{cases}
$$ 
Note that $[b]$ is a singleton for each $b \in C$ and thus on the right-hand side of the definitions of $g$ and $h$ we can indeed safely write $b$. The poset $(D, \leq)$ is a cpo. Indeed, suppose that $Y$ is a chain in $D$. 
\begin{itemize}
\item If $B_1 \cap Y$ is cofinal in $Y$, then let $X = \pi_1(B_1 \cap Y)$. In this way we obtain a chain $X$ in $B$ such that $g(X) = B_1 \cap Y$. If $\sup X \in B \setminus C$, then the supremum of $g(X)$ exists and equals $[(\sup X, 1)]$. Observe that the last condition is used here. If $\sup X \in C$, then the supremum of $g(X)$ exists and equals $\sup X$. Therefore, the supremum of $B_1 \cap Y$ exists in $D$, and thus the supremum of $Y$ exists in $D$ and equals the supremum of $B_1 \cap Y$ in $D$.
\item Analogously, if $B_2 \cap Y$ is cofinal in $Y$, then $\sup Y$ exists in $D$. 
\item If $C \cap Y$ is cofinal in $Y$, then the supremum of $Y$ exists in $D$ and is identical to the supremum $s$ of $C \cap Y$ in $B$, since $s$ belongs to $C$ because it is a directed join of elements from $C$. 
\end{itemize}
Since at least one of the sets $B_1 \cap Y$, $B_2 \cap Y$, $C \cap Y$ must be cofinal in $Y$, we see that $D$ is a cpo.\par
Moreover, the maps $g$ and $h$ are cpo maps. Let us show it for $g$, the argument for $h$ is completely analogous. Suppose that $X = \{b_j \mid j \in J\}$ is a chain in $B$. We wish to show that $g(\sup X) = \sup g(X)$. Note that $g$ is obviously isotone.
\begin{itemize}
\item If $C \cap X$ is not cofinal in $X$, then there exists $k \in J$ such that each ${b_\ell \geq b_k}$ satisfies $g(b_\ell) = [(b_\ell, 1)]$, in other words, the chain $g(X)$ is eventually contained completely in $B_1$. Now note that if $\sup X$ belongs to $B \setminus C$, then ${g(\sup X) = [(\sup X, 1)] = \sup g(X)}$, and if $\sup X$ belongs to $C$, then $g(\sup X) = \sup X = \sup g(X)$.
\item If $C \cap X$ is cofinal in $X$, then the supremum of $g(X)$ exists in $D$ and is identical to the supremum $s$ of $g(C \cap X) = C \cap X$ in $B$, since $s$ belongs to $C$ because it is a directed join of elements from $C$. Moreover, the latter also implies that $g(s) = s$. Thus,\newline
$g(\sup X) = g(\sup (C \cap X)) = \sup (C \cap X) = \sup g(C \cap X) = \sup g(X).$
\end{itemize}
To summarize, we've shown that $D$ is a cpo and $g, h$ are cpo maps. From the explicit description of $g$ and $h$ and the fact that $C$ is closed in $B$ under all joins wee see that $g \neq h$ if $C \neq B$, since they differ on each element of $B \setminus C$. Indeed, for the sake of contradiction, suppose that there exists $b \in B \setminus C$ such that $[(b,1)] = [(b,2)]$, i.e.\ $(b,1) \lesssim (b,2)$ and $(b,2) \lesssim (b,1)$. Hence there exists $b' \geq_B b$ and a subset $Z \subseteq B \setminus C$ such that $b' \in \operatorname{dcl}(Z)$ in $B$, and a set $\{c_z \in C \mid z \in Z\}$ of elements such that $z \leq_B c_z \leq_B b$ for each $z \in Z$. We will show that $b$ is the least upper bound of $\{c_z \in C \mid z \in Z\}$. Clearly, it is an upper bound. Now suppose that $x \in B$ is an upper bound, i.e. $c_z \leq_B x$ for each $z \in Z$. Therefore, $z \leq_B x$ for all $z \in Z$, and thus by transfinite induction and by $b' \in \operatorname{dcl}(Z)$ we obtain $b' \leq_B x$, which implies $b \leq_B x$. We conclude $b = \sup_{z \in Z} c_z$, which is in contradiction with $C$ being closed in $B$ under all joins.
\end{proof}

\begin{rem}
\begin{enumerate}[(1)]
\item Epimorphisms do not need to be dense, see \cite[Theorem 2]{LP}. Note that the result in \cite{LP} is stated for $\omega$-complete posets, however, as is pointed out in \cite[p.~1]{LP} the same counter-example works also for chain-complete posets.
\item Evidently, join dense maps do not need to be epimorphisms.

\item  Every cpo map $f \colon A \to B$ has a factorization ${f = m \cdot e}$ through $C = \operatorname{dcl}(f[A])$, where the partial order on the cpo $C$ is inherited from $B$, ${m \colon C \to B}$ is the inclusion, and $e \colon A \to C$ is the codomain-restriction of $f$. This factorization was observed in \cite{GL}.
\end{enumerate}
\end{rem}

\begin{propo}\label{prop_cowell}
$\CPO$ is co-wellpowered.
\end{propo}
\begin{proof}
This follows from the second implication in Lemma \ref{lem-epi}, since for an epimorphism $f \colon A \to B$ in $\CPO$ we have $|B|\leq 2^{|A|}$ because $\operatorname{cl}(f[A]) = B$.
\end{proof}

\begin{rem} 
The results obtained above allow us to conclude that the category $\mathbf{CPO}$ is nearly locally $\aleph_0$-presentable. Indeed, it is cocomplete, co-wellpowered, and the ordinal $3$ is a nearly $\aleph_0$-presentable object that forms a strong generator in $\mathbf{CPO}$. 
\end{rem}

\section{Strong epimorphisms in $\CPO$}
In this section we characterize strong epimorphisms in $\CPO$.
\begin{propo}\label{str_epi_char}
A cpo map $f \colon A \to B$ is a strong epimorphism in $\CPO$ if and only if it satisfies
\begin{enumerate}[(i)]
\item $\operatorname{dcl}(f[A]) = B$,
\item if $b_1, b_2 \in f[A]$ and $b_1 \leq b_2$, then there exist elements $a \in A$, $a' \in A$, an ordinal $\gamma$, and sets $\{a_\delta \in A \mid \delta \leq \gamma\}$, $\{a'_\delta \in A \mid \delta \leq \gamma\}$ such that $f(a) = b_1$, $f(a') = b_2$, $a \leq a_0$, $a'_\gamma \leq a'$, $a'_\delta \leq a_{\delta + 1}$ for each ordinal $\delta < \gamma$, $f(a_\delta) = f(a'_{\delta})$ for each ordinal $\delta \leq \gamma$, and for each limit ordinal $\delta \leq \gamma$ there exists an ordinal $\xi < \delta$ such that $a'_\xi \leq a_\delta$, and
\item if $b \leq \bigvee_{i \in I} b_i$, where $\alpha$ is an ordinal, $b \in f[A]_\alpha$, the set $\{b_i \in f[A]_\alpha \mid i \in I\}$ is directed in $f[A]_\alpha$, and $\bigvee_{i \in I} b_i \not\in f[A]_\alpha$, then there exists $j \in I$ such that $b \leq b_j$.
\end{enumerate}
\end{propo}
\begin{proof}
First, note that since $\CPO$ has pullbacks \cite[Theorem 2.6]{M}, we know that strong epimorphisms and extremal epimorphisms coincide \cite[0.5]{AR}.

Suppose that the conditions $(i)$, $(ii)$, $(iii)$ hold. Let $A \xrightarrow{g} C \xrightarrow{m} B$ be a factorization of $f$ such that $m$ is a monomorphism. Since extremal epimorphisms and strong epimorphisms coincide, it suffices to show that $m$ is surjective and order-reflecting, and then the proof of the fact that $f$ is a strong epimorphism will be finished. Let us proceed by transfinite induction.
\begin{enumerate}[(a)]
\item If $b \in f[A]_0$, then there obviously exists $a \in A$ such that $$m(g(a)) = f(a) = b.$$ Furthermore, we will show that for each $b_1 \in f[A]_0$, $b_2 \in f[A]_0$, the inequality $b_1 \leq b_2$ implies $m^{-1}(b_1) \leq m^{-1}(b_2)$. Indeed, by assumption there exist elements $a \in A$, $a' \in A$, an ordinal $\gamma$, and sets $\{a_\delta \in A \mid \delta \leq \gamma\}$, ${\{a'_\delta \in A \mid \delta \leq \gamma\}}$ satisfying $(ii)$. Since $m$ is a monomorphism, we know that $g(a_\delta) = g(a'_\delta)$ for each $\delta \leq \gamma$. Thus, we get that $m(g(a)) = b_1$, $m(g(a')) = b_2$, $g(a) \leq g(a_0)$, $g(a_\gamma) \leq g(a')$, $g(a_\delta) \leq (a_{\delta + 1})$ for each ordinal $\delta < \gamma$, and for each limit ordinal $\delta \leq \gamma$ there exists an ordinal $\xi < \delta$ such that $g(a_\xi) \leq g(a_\delta)$. Using transfinite induction with respect to the ordinal $\delta$ it is easy to see that $g(a) \leq g(a_\delta)$ for each ordinal $\delta \leq \gamma$. Indeed, we know that $g(a) \leq g(a_0)$. If we know that $g(a) \leq g(a_\delta)$ for some $\delta < \gamma$, then, using transitivity and the fact that $g(a_\delta) \leq g(a_{\delta + 1})$, we obtain $g(a) \leq g(a_{\delta + 1})$. And finally, if $\delta \leq \gamma$ is a limit ordinal and we know that $g(a) \leq g(a_\xi)$ for all $\xi < \delta$, then we use transitivity and the fact that we know that there exists $\xi < \delta$ such that $g(a_\xi) \leq g(a_\delta)$ to obtain $g(a) \leq g(a_\delta)$. This finishes the proof by transfinite induction, and thus we get that $g(a) \leq g(a_\gamma) \leq g(a')$.
\item Let $\alpha$ be an ordinal and suppose that we've already found preimages of all elements of $f[A]_\alpha$ under $m$ and that for each $b_1 \in f[A]_\alpha$, $b_2 \in f[A]_\alpha$, the inequality $b_1 \leq b_2$ implies $m^{-1}(b_1) \leq m^{-1}(b_2)$. Consider an element $b \in f[A]_{\alpha + 1}$, i.e.\ $b = \bigvee_{i \in I} b_i$, where each $b_i \in f[A]_\alpha$ and the set $\{b_i \mid i \in I\}$ is directed in $f[A]_\alpha$. The set $\{m^{-1}(b_i) \mid i \in I\}$ is directed because by inductive hypothesis $m^{-1}|_{f[A]_\alpha}$ preserves partial order. Therefore $$m\Big(\bigvee_{i \in I} m^{-1}(b_i)\Big) = \bigvee_{i \in I} m(m^{-1}(b_i)) = \bigvee_{i \in I} b_i = b.$$ Now suppose that $b_1 \in f[A]_{\alpha + 1}$, $b_2 \in f[A]_{\alpha + 1}$, and $b_1 \leq b_2$. This means that $b_1 = \bigvee_{i \in I} b_i$, $b_2 = \bigvee_{j \in J} b_j$, where each $b_i \in f[A]_\alpha$, each $b_j \in f[A]_\alpha$, and the sets $\{b_i \mid i \in I\}$, $\{b_j \mid j \in J\}$ are directed in $f[A]_\alpha$. For each $i \in I$ we know that $b_i \leq \bigvee_{j \in J} b_j$. If $\bigvee_{j \in J} b_j \in f[A]_\alpha$, then the inductive hypothesis implies the following inequality $m^{-1}(b_i) \leq m^{-1}(\bigvee_{j \in J} b_j) = \bigvee_{j \in J} m^{-1}(b_j)$, which gives us that $\bigvee_{i \in I} m^{-1}(b_i) \leq \bigvee_{j \in J} m^{-1}(b_j)$. If $\bigvee_{j \in J} b_j \not\in f[A]_\alpha$, then using $(iii)$ we get that for each $i \in I$ there exists $j_i \in J$ such that $b_i \leq b_{j_i}$. By inductive hypothesis we know that this implies $m^{-1}(b_i) \leq m^{-1}(b_{j_i}) \leq \bigvee_{j \in J} m^{-1}(b_j)$. Hence, $\bigvee_{i \in I} m^{-1}(b_i) \leq \bigvee_{j \in J} m^{-1}(b_j).$ Furthermore, note that $$m\Big(\bigvee_{i \in I} m^{-1}(b_i)\Big) = \bigvee_{i \in I} m(m^{-1}(b_i)) = \bigvee_{i \in I} b_i = b_1,$$ and completely analogously $m(\bigvee_{j \in j} m^{-1}(b_j)) = b_2$.
\item Let $\beta$ be a limit ordinal and suppose that for each ordinal $\alpha < \beta$ we've already found preimages of all elements of $f[A]_\alpha$ under $m$ and that for each $b_1 \in f[A]_\alpha$, $b_2 \in f[A]_\alpha$, the inequality $b_1 \leq b_2$ implies $m^{-1}(b_1) \leq m^{-1}(b_2)$. Consider an element $b \in f[A]_{\beta}$, then we know that $b \in f[A]_{\alpha}$ for some ordinal $\alpha < \beta$, and thus by inductive hypothesis we have a preimage of $b$ under $m$. Now suppose that $b_1 \in f[A]_\beta$, $b_2 \in f[A]_\beta$, and $b_1 \leq b_2$. Then there exist ordinals $\alpha_1 < \beta$, $\alpha_2 < \beta$ such that $b_1 \in f[A]_{\alpha_1}$, $b_2 \in f[A]_{\alpha_2}$. Take $\alpha := \max\{\alpha_1, \alpha_2\}$. Then $b_1 \in f[A]_\alpha$, $b_2 \in f[A]_\alpha$, $\alpha < \beta$, and thus by the inductive hypothesis we know that $m^{-1}(b_1) \leq m^{-1}(b_2)$.
\end{enumerate}

Now suppose that $f \colon A \to B$ is a strong epimorphism. Note that each strong epimorphism is dense, since each cpo embedding is a monomorphism. Therefore, $\operatorname{dcl}(f[A]) = B$. Let us now denote $e_0 := f$, $A_0 := A$. Perform the transfinite construction from the proof of \cite[Lemma 4.2]{PR} for $e_0 \colon A_0 \to B$. Since $\CPO$ is co-wellpowered, we know that the transfinite construction stops. Using transfinite induction it is now easy to see that $e_0$ satisfies $(ii)$ and $(iii)$.
\end{proof}
\begin{coro}
If a cpo map $f \colon A \to B$ satisfies
\begin{enumerate}[(i)]
\item $\overline{f[A]} = B$, and
\item if $b_1, b_2 \in f[A]$, and $b_1 \leq b_2$, then there exist $a_1 \in A$, $a_2 \in A$ such that $f(a_1) = b_1$, $f(a_2) = b_2$, and $a_1 \leq a_2$,
\item if $b \leq \bigvee_{i \in I} b_i$, where $\alpha$ is an ordinal, $b \in f[A]_\alpha$, the set $\{b_i \in f[A]_\alpha \mid i \in I\}$ is directed in $f[A]_\alpha$, and $\bigvee_{i \in I} b_i \not\in f[A]_\alpha$, then there exists $j \in I$ such that $b \leq b_j$,
\end{enumerate}
then $f \colon A \to B$ is a strong epimorphism in $\CPO$.
\end{coro}
\begin{proof}
This follows immediately from Proposition \ref{str_epi_char}.
\end{proof}


\begin{thebibliography}{EAPT}
\itemsep=2pt
 
\bibitem{AR} J. Ad\' amek and J. Rosick\' y, {\em Locally Presentable 
and Accessible Categories}, Cambridge University Press 1994.

\bibitem{GL} J. Goubault-Larrecq. {\em Reading notes: Why is cpo cocomplete?}, \url{www.lsv.ens-cachan.fr/Publis/RAPPORTS_LSV/PS/rr-lsv-2002-15.rr.ps}, accessed: 08-20-2018.

\bibitem{JR} J. Jurka and J. Rosick\' y, {\em Are chain-complete posets co-wellpowered?}, Order 39 (2022), 71-76.

\bibitem{JR_err} J. Jurka and J. Rosick\' y, {\em Erratum to "Are chain-complete posets co-wellpowered?" [Order 39 (2022), 71-76]}, Order 42 (2025), 251-252.

\bibitem{LP} D. Lehmann and A. Pasztor, {\em Epis need not be dense},
Th. Comp. Sci. 17 (1982), 151-161.

\bibitem{M} G. Markowsky, {\em Categories of chain-complete posets}, Th. Comp. Sci. 4 (1977) 125-135.

\bibitem{P} A. Pasztor, {\em The epis of POS(Z)},  Comment. Math. Univ. Carol. 23 (1982), 285-299.

\bibitem{PR}{L. Positselski, J. Rosick\'y}, {\it Nearly locally presentable
	 	categories}, Th. Appl. Categ. 33 (2018), 253--264.

\end{thebibliography}
\end{document}